\documentclass{amsart}

\usepackage[bookmarksnumbered,colorlinks, plainpages]{hyperref}
\hypersetup{colorlinks=true,linkcolor=red, anchorcolor=green, citecolor=cyan, urlcolor=red, filecolor=magenta, pdftoolbar=true}

\usepackage{amsmath, amsthm, amscd, amsfonts, amssymb, graphicx, color, enumerate, xcolor, mathrsfs, latexsym}

\usepackage{tikz}

\theoremstyle{plain} 
\newtheorem{theorem}{Theorem}[section]

\newtheorem{lemma}[theorem]{Lemma}
\newtheorem{corollary}[theorem]{Corollary}
\newtheorem{conjecture}[theorem]{Conjecture}

\theoremstyle{definition}

\newtheorem*{question}{Question}

\theoremstyle{remark}

\renewcommand{\leq}{\leqslant} 
\renewcommand{\geq}{\geqslant}

\newcommand{\parn}[1]{\left(#1\right)}

\newcommand{\ceil}[1]{\left\lceil#1\right\rceil}
\newcommand{\floor}[1]{\left\lfloor#1\right\rfloor}
\newcommand{\bracket}[1]{\left[#1\right]}

\newcommand{\C}{\mathcal{C}}
\newcommand{\E}{\mathcal{E}}
\newcommand{\F}{\mathcal{F}}
\newcommand{\I}{\mathcal{I}}
\newcommand{\M}{\mathcal{M}}
\renewcommand{\O}{\mathcal{O}}
\newcommand{\T}{\mathcal{T}}
\newcommand{\Z}{\mathbb{Z}}

\begin{document}
\title[Lattice paths inside a table]{Lattice paths inside a table: Rows and columns linear combinations}
\author{M. Farrokhi D. G.} 

\address{Department of Mathematics, Institute for Advanced Studies in Basic Sciences (IASBS), and the Center for Research in Basic Sciences and Contemporary Technologies, IASBS, Zanjan 66731-45137, Iran}
\email{m.farrokhi.d.g@gmail.com\\farrokhi@iasbs.ac.ir}
\date{}

\begin{abstract}
A lattice path inside the $m\times n$ table $T$ is a sequence $\nu_1,\ldots,\nu_k$ of cells such that $\nu_{j+1}-\nu_j\in\{(1,-1),(1,0),(1,1)\}$ for all $j=1,\ldots,k-1$. The number of lattice paths in $T$ from the first column to the $(x,y)$-cell is written into that cell. We present a precise description of the minimal linear recurrences among rows, columns, and columns sums. As a result, we obtain several formulas for the number of all lattice paths from the first column to the last column of $T$, that is, the $n^{th}$ column sum. Our methods are based on three classes of operators, which will also be studied independently.
\end{abstract}
\maketitle
%==================================================
\section{Introduction}
A path through the points of a set $U\subseteq\Z^d$ is known as a \textit{lattice path}, where by a path we mean a sequence of points. In particular, for a set $S\subseteq\Z^d$ of steps, an \textit{$S$-lattice path} inside $U$ is a sequence $\nu_1,\ldots,\nu_n\in U$ of points such that every step $\nu_{i+1}-\nu_i$ belongs to $S$, for all $i=1,\ldots,n-1$. Lattice paths are studied from a general (analytical) point of view by various authors where they investigate analytic behavior of complex generating functions of paths, estimations of the number of paths of given lengths, etc. (see \cite{cb-pf, mb}). However, a large variety of $S$-lattice paths inside a set $U$ of points are studied in the literature too, for instance Dyke paths, Schr\"{o}der paths, Delannoy paths, Motkzin paths, Narayana paths etc. Motzkin paths, the most related lattice paths to ours, are well studied in \cite{frb, ck-rwm, spe-scl-yny}. 

Every rectangular subset $U$ of $\Z^2$ can be considered as a table $T$ whose cells are correspond to the points of $U$. For a given $m\times\infty$ table $T_m$, we may consider $S$-lattice paths, where $S=\{(1,-1),(1,0),(1,1)\}$. In particular, lattice paths starting from a cell in the first column are of special interest. The number of all $S$-lattice paths from the first column to the $(x,y)$-cell is denoted by $\C_m(x,y)$, or simply by $\C(x,y)$ if there is no confusion. We usually write the number $\C(x,y)$ inside the $(x,y)$-cell for convenience (see Figure 1). Assuming $\C(x,0)=\C(x,m+1)=0$ for all $x\geq1$, we observe that
\[\C(x+1,y)=\C(x,y-1)+\C(x,y)+\C(x,y+1)\]
for all $x\geq1$ and $y=1,\ldots,m$.
\begin{figure}[h!]
\begin{tikzpicture}[scale=0.5]
\draw (0,7)--(6.0,7);
\draw (0,6)--(5.5,6);
\draw (0,5)--(4.5,5);
\draw (0,4)--(3.5,4);
\draw (0,2)--(3.5,2);
\draw (0,1)--(4.5,1);
\draw (0,0)--(5.5,0);

\draw (0,3.5)--(0,7);
\draw (1,3.5)--(1,7);
\draw (2,4.0)--(2,7);
\draw (3,4.0)--(3,7);
\draw (4,4.5)--(4,7);
\draw (5,5.5)--(5,7);
\draw (0,0)--(0,2.5);
\draw (1,0)--(1,2.5);
\draw (2,0)--(2,2.0);
\draw (3,0)--(3,2.0);
\draw (4,0)--(4,1.5);
\draw (5,0)--(5,0.5);

\node at (0.5,0.5) {$1$};
\node at (0.5,1.5) {$1$};
\node at (1.5,0.5) {$2$};
\node at (1.5,1.5) {$3$};
\node at (2.5,0.5) {$5$};
\node at (2.5,1.5) {$8$};
\node at (3.5,0.5) {$13$};

\node at (0.5,6.5) {$1$};
\node at (0.5,5.5) {$1$};
\node at (0.5,4.5) {$1$};
\node at (1.5,6.5) {$2$};
\node at (1.5,5.5) {$3$};
\node at (1.5,4.5) {$3$};
\node at (2.5,6.5) {$5$};
\node at (2.5,5.5) {$8$};
\node at (2.5,4.5) {$9$};
\node at (3.5,6.5) {$13$};
\node at (3.5,5.5) {$22$};
\node at (4.5,6.5) {$35$};

\draw (7.0,7)--(15.5,7);
\draw (7.5,6)--(15.0,6);
\draw (8.3,5)--(14.5,5);
\draw (8.5,4)--(14.0,4);
\draw (8.0,3)--(12.0,3);
\draw (8.5,1)--(10.5,1);
\draw (8.0,0)--(12.0,0);

\draw (9.0,0)--(9.0,1.5);
\draw (10.0,0)--(10.0,1.5);
\draw (11.0,0)--(11.0,0.5);
\draw (8.0,5.5)--(8.0,7);
\draw (9.0,2.5)--(9.0,7);
\draw (10.0,2.5)--(10.0,7);
\draw (13.5,3.5)--(13.5,7);
\draw (14.5,5.5)--(14.5,7);

\node at (9.5,6.5) {$x$};
\node at (9.5,5.5) {$y$};
\node at (9.5,4.5) {$z$};
\node at (9.5,3.5) {$w$};
\node at (11.75,6.5) {$x+y$};
\node at (11.75,5.5) {$x+y+z$};
\node at (11.75,4.5) {$y+z+w$};
\node at (9.5,0.5) {$x$};

\node at (0,3.25) {$\vdots$};
\node at (1,3.25) {$\vdots$};
\node at (9,2.25) {$\vdots$};
\node at (10,2.25) {$\vdots$};
\node at (6.5,7) {$\ldots$};
\node at (6.5,6) {$\ldots$};
\node at (6.75,0) {$\ldots$};
\node at (16.0,7) {$\ldots$};
\node at (15.5,6) {$\ldots$};
\node at (15.0,5) {$\ldots$};
\node at (14.5,4) {$\ldots$};
\node at (11.0,1) {$\ldots$};
\node at (12.5,0) {$\ldots$};
\end{tikzpicture}\\
\caption{The $m\times\infty$ table}
\end{figure}
The number of lattice paths from the first column to the $n^{th}$ column is denoted by $\I_m(n)$. It is evident that
\[\I_m(n)=\C(n,1)+\cdots+\C(n,m).\]

Lattice paths inside a table are studied in \cite{dy-mfdg-hgz, dy-mfdg} in several special cases, say when $m=1,2,3,4$, and some formulas for $\C(x,y)$ and $\I_m(n)$ are derived. It is shown that the number of paths from the first column to the cell $(n,1)$ or $(n,n)$ in an square $n\times n$ table, namely $\C(n,1)=\C(n,n)$, counts the number of directed animals of size $n$. Directed animals appear frequently in physics in study of thermodynamic models for critical phenomena, phase transitions, statistical physics, lattice gas models with extended hard-cores, river networks, etc (see \cite{dg-gv} and reference therein). Most results there are deal with exact formulas or asymptotic results for the number of various kinds of directed animals in $d$-dimensional spaces. For instance, it is known that $a_n^{1/n}$ tends to a constant in many cases, where $a_n$ denotes the number of directed animals of size $n$.

Recently, determinants of some Hankel matrices involving the numbers $\C_m(x,y)$ and their (weighted) generalizations are also computed in \cite{ck-dy}.

The aim of this paper is to study linear combination of rows and columns of the table and obtain recurrence relations of minimum degrees for $\C(x,y)$ and $\I_m(n)$. We note that the degree of a linear recurrence relation
\[a(n+k)=\alpha_0 a(n)+\cdots+\alpha_{k-1}a(n+k-1),\]
a sequence $\{a(n)\}$ satisfies is the number $k$ provided that $\alpha_0\neq0$. Indeed, $k$ is the degree of the associated polynomial of the corresponding recurrence relation.

Our results uses three classes of operators, denoted by $\M_o$, $\M_e$, and $\M'$, defined recursively. In section 2, the main results (say, Theorem \ref{rows linear combination}, and Theorem \ref{columns linear combination} and subsequent corollaries) are proved. Theorem \ref{rows linear combination} determines under which conditions a linear combination of rows (or rows entries) is a constant. Theorem \ref{columns linear combination} gives the recurrence relations of minimum degrees to which columns and columns sums satisfy. Since our methods make use of the operators $\M_o$, $\M_e$, and $\M'$ heavily, in the last section, we shall also describe their properties as well as giving precise formulas of them.
%==================================================
\section{Linear recurrences among rows, columns, and columns sums}
We begin this section with analyzing linear combinations of rows. Since the rows in an $m\times\infty$ table are symmetric, we always have equations of the form
\[\alpha_1\C(n,1)+\cdots+\alpha_m\C(n,m)=0,\]
where $\alpha_i+\alpha_{m+1-i}=0$ for all $i=0,\ldots,m$. Such a linear combination of columns entries (or rows), appearing in left hand side of the above equation, are called trivial linear combinations.
%-----------------------------------
\begin{theorem}\label{rows linear combination}
Inside the $m\times\infty$ table, a nontrivial linear combination of columns entries is a constant if and only if $m\equiv1\pmod4$, and the equation is given by
\[\alpha_1\C(n,1)+\alpha_3\C(n,3)+\cdots+\alpha_{m-2}\C(n,m-2)+\alpha_m\C(n,m)=\lambda\]
for all $n\geq1$, where $\lambda\neq0$ is a fixed number and $\alpha_{2i+1}+\alpha_{m-2i}=(-1)^i2\lambda$, for all $i=0,\ldots,(m-1)/4$.
\end{theorem}
\begin{proof}
Suppose
\begin{equation}\label{rows linear combination 1}
\alpha_1\C(n,1)+\cdots+\alpha_m\C(n,m)=c
\end{equation}
is a constant for all $n\geq1$. We know that 
\[\C(n+1,i)=\begin{cases}
\C(n,i)+\C(n,i+1),&i=1,\\
\C(n,i-1)+\C(n,i)+\C(n,i+1),&1<i<m,\\
\C(n,i-1)+\C(n,i),&i=m
\end{cases}\]
for all $1<i<m$. Now from 
\[\alpha_1\C(n+1,1)+\cdots+\alpha_m\C(n+1,m)=c\]
we get
\begin{equation}\label{rows linear combination 2}
(\alpha_1+\alpha_2)\C(n,1)+\sum_{i=2}^{m-1}(\alpha_{i-1}+\alpha_i+\alpha_{i+1})\C(n,i)+(\alpha_{m-1}+\alpha_m)\C(n,m)=c.
\end{equation}
Combining \eqref{rows linear combination 1} and \eqref{rows linear combination 2} yields
\begin{equation}\label{rows linear combination 3}
\alpha_2\C(n,1)+\sum_{i=2}^{m-1}(\alpha_{i-1}+\alpha_{i+1})\C(n,i)+\alpha_{m-1}\C(n,m)=0
\end{equation}
for all $n\geq1$. Let $\beta_1=\alpha_2$, $\beta_m=\alpha_{m-1}$, and $\beta_i=\alpha_{i-1}+\alpha_{i+1}$ for all $1<i<m$. Then
\[\beta_1\C(n,1)+\cdots+\beta_m\C(n,m)=0\]
for all $n\geq1$. We show that $\beta_i+\beta_{m+1-i}=0$ for all $i\leq m/2$. The cases $n=1$ and $n=2$ yield $\sum_{i=1}^m\beta_i=3\sum_{i=1}^m\beta_i-(\beta_1+\beta_m)=0$, which imply that $\beta_1+\beta_m=0$. Now assume that $\beta_i+\beta_{m+1-i}=0$ for all $i=1,\ldots,k<m/2$. Then 
\begin{equation}\label{rows linear combination 4}
\beta_{k+1}\C(n,k+1)+\cdots+\beta_{m-k}\C(n,m-k)=0
\end{equation}
for all $n\geq1$. Considering the equation \eqref{rows linear combination 3} as a transformation of \eqref{rows linear combination 1} with $c=0$, and applying it $k$ times to the equation \eqref{rows linear combination 4}, we obtain the equation
\[\beta_{k+1}\C(n,1)+\beta'_2\C(n,2)+\cdots+\beta'_{m-1}\C(n,m-1)+\beta_{m-k}\C(n,m)=0\]
for some $\beta'_2,\ldots,\beta'_{m-1}$ and all $n\geq1$. Proceeding the same argument as above, we get $\beta_{k+1}+\beta_{m-k}=0$. Hence, we have shown that $\beta_i+\beta_{m+1-i}=0$ for all $i\leq m/2$. Therefore,
\[\begin{cases}
\alpha_2+\alpha_{m-1}=0,&\\
\alpha_i+\alpha_{i+2}+\alpha_{m-1-i}+\alpha_{m+1-i}=0,&1\leq i\leq m/2.
\end{cases}\]
If either $m$ is even or $m\equiv3\pmod4$, then we get $\alpha_i+\alpha_{m+1-i}=0$ for all $i\leq m/2$, which results in a trivial equation with zero constant. Now assume that $m\equiv1\pmod4$ is odd. Then $\alpha_{2i}+\alpha_{m+1-2i}=0$, for $i=1,\ldots,(m-1)/4$, and there exists a number $\lambda$ such that $\alpha_{2i+1}+\alpha_{m-2i}=(-1)^i2\lambda$ for all $i=0,\ldots,(m-1)/4$. Since $\lambda=0$ yields a trivial equation, we must have $\lambda\neq0$, as required. 

To prove the converse, let $\lambda\neq0$ be any number and assume $\alpha_{2i+1}+\alpha_{m-2i}=(-1)^i2\lambda$ for all $i=0,\ldots,(m-1)/4$. Clearly, the equation
\[\alpha_1\C(n,1)+\alpha_3\C(n,3)+\cdots+\alpha_{m-2}\C(n,m-2)+\alpha_m\C(n,m)=\lambda\]
holds for $n=1$. Now a simple inductive argument on $n$ in conjunction with \eqref{rows linear combination 2} shows that the equation holds for all $n\geq1$. The proof is complete.
\end{proof}

The rest of this section is devoted to the study of linear combination of rows entries. In order to do this, we define three classes of operators and apply them to find the linear recurrences among rows entries, columns and columns sums. Let $\M_m:=\M_m(\Delta)$ be the \textit{multiplier function} defined as 
\[\M_m(0)=\begin{cases}2,&m\text{ is odd},\\1,&m\text{ is even},\end{cases}\]
$\M_m(1)=\Delta-2+\M_m(0)$, and $\M_m(n+2)=\Delta\M_m(n+1)-\M_m(n)$ for all $n\geq0$, where $\Delta$ is the difference operator defined as $\Delta a(n)=a(n+1)-a(n)$ for any sequence $\{a(n)\}$ of numbers. For convenience, we set $\M_o:=\M_1$ and $\M_e:=\M_2$ as $\M_m$ depends only on the parity of $m$.

In what follows, we shall use the multipliers $\M_m$, acting on the second argument of $\C(n,i)$, to obtain relations for columns entries and apply them to derive formulas for $\I_m(n)$ as a function of a column entry.
%-----------------------------------
\begin{lemma}\label{two rows equivalence by M}
Inside the $m\times\infty$ table, we have
\[\M_m(k-b)\C(n,a)=\M_m(k-a)\C(n,b)\]
for all $a,b=1,\ldots,k$, where $k=\ceil{m/2}$.
\end{lemma}
\begin{proof}
Let $k=\ceil{m/2}$. First we show that
\[\M_m(a)\C(n,k)=\M_m(0)\C(n,k-a)\]
for all $a=0,\ldots,k-1$. We know from the definition that 
\[\Delta\C(n,a)=\C(n,a-1)+\C(n,a+1)\]
for all $1<a<m$. We have two cases:

(1) $m$ is odd. Then $m=2k+1$. We have 
\[\M_m(1)\C(n,k)=\C(n,k-1)+\C(n,k+1)=2\C(n,k-1).\]
Also, $\Delta\M_m(1)=2\C(n,k-2)+2\C(n,k)$, which implies that 
\[\M_m(2)\C(n,k)=2\C(n,k-2).\]
Now if the result holds for some $2\leq a<k$, then $\M_m(a)\C(n,k)=2\C(n,k-a)$, from which it follows that 
\begin{align*}
\Delta\M_m(a)\C(n,k)=&2\C(n,k-(a+1))+2\C(n,k-(a-1))\\
=&\M_m(a-1)\C(n,k)+2\C(n,k-(a+1)).
\end{align*}
Thus,
\[\M_m(a+1)\C(n,k)=2\C(n,k-(a+1)),\]
and the result follows.

(2) $m$ is even. Then $m=2k$. In this case 
\[\Delta\C(n,k)=\C(n,k-1)+\C(n,k+1)=\C(n,k-1)+\C(n,k),\]
that is,
\[\M_m(1)\C(n,k)=\C(n,k-1).\]
Also, $\Delta\M_m(1)\C(n,k)=\C(n,k-2)+\C(n,k)$ so that
\[\M_m(2)\C(n,k)=\C(n,k-2).\]
The rest of proof is similar to (1) and we are done.

Now let $0\leq a,b\leq k-1$. Then
\[\M_m(a)\C(n,k)=\M_m(0)\C(n,k-a)\]
and
\[\M_m(b)\C(n,k)=\M_m(0)\C(n,k-b)\]
Therefore,
\begin{align*}
\M_m(0)\M_m(b)\C(n,k-a)&=\M_m(a)\M_m(b)\C(n,k)\\
&=\M_m(0)\M_m(a)\C(n,k-b),
\end{align*}
from which, by substituting $a\mapsto k-a$ and $b\mapsto k-b$, the result follows.
\end{proof}
%-----------------------------------
\begin{corollary}\label{I_m(n) by D(n,a) and M}
Inside the $m\times\infty$ table, we have
\[\M_m(k-a)\I_m(n)=2(\M_m(k-1)+\cdots+\M_m(1)+1)\C(n,a)\]
for all $a=1,\ldots,k$ and $n\geq1$, where $k=\ceil{m/2}$. In particular, for $a=k$, we have
\[\I_m(n)=\frac{2}{\M_m(0)}(\M_m(k-1)+\cdots+\M_m(1)+1)\C(n,k)\]
for all $n\geq1$.
\end{corollary}
\begin{proof}
We know that $\M_m(k-a)\C(n,b)=\M_m(k-b)\C(n,a)$ for all $b=1,\ldots,k$. If $m$ is odd, then
\begin{align*}
\M_m(k-a)\I_m(n)&=\M_m(k-a)(2\C(n,1)+\cdots+2\C(n,k-1)+\C(n,k))\\
&=(2\M_m(k-1)+\cdots+2\M_m(1)+\M_m(0))\C(n,a)\\
&=2(\M_m(k-1)+\cdots+\M_m(1)+1)\C(n,a).
\end{align*}
Also, if $m$ is even, then
\begin{align*}
\M_m(k-a)\I_m(n)&=\M_m(k-a)(2\C(n,1)+\cdots+2\C(n,k-1)+2\C(n,k))\\
&=(2\M_m(k-1)+\cdots+2\M_m(1)+2\M_m(0))\C(n,a)\\
&=2(\M_m(k-1)+\cdots+\M_m(1)+1)\C(n,a).
\end{align*}
The proof is complete.
\end{proof}

Analogous to the multipliers $\M_m$, we may define the multipliers $\M'$ as $\M'(0)=1$, $\M'(1)=\Delta$, and $\M'(n+2)=\Delta\M'(n+1)-\M'(n)$ for all $n\geq0$.
%-----------------------------------
\begin{lemma}\label{two rows equivalence by M'}
Inside the $m\times\infty$ table, we have
\[\M'(b-1)\C(n,a)=\M'(a-1)\C(n,b)\]
for all $a,b=1,\ldots,k$, where $k=\ceil{m/2}$.
\end{lemma}
\begin{proof}
Let $k=\ceil{m/2}$. First we show that
\[\M'(a-1)\C(n,1)=\C(n,a)\]
for all $i=1,\ldots,k$. Clearly, $\M'(0)\C(n,1)=\C(n,1)$. Also, $\Delta\C(n,1)=\C(n,2)$, that is, $\M'(1)\C(n,1)=\C(n,2)$. Now assume $a<k$ and $\M'(b-1)\C(n,1)=\C(n,b)$ for all $1\leq b\leq a$. Then 
\begin{align*}
\Delta\M'(a-1)\C(n,1)&=\Delta\C(n,a)\\
&=\C(n,a-1)+\C(n,a+1)\\
&=\M'(a-2)\C(n,1)+\C(n,a+1),
\end{align*}
from which it follows that $\M'(a)\C(n,1)=\C(n,a+1)$.

Now let $1\leq a,b\leq k$. Then
\[\M'(a-1)\C(n,1)=\C(n,a)\]
and
\[\M'(b-1)\C(n,1)=\C(n,b).\]
Therefore,
\[\M'(b-1)\C(n,a)=\M'(a-1)\M'(b-1)\C(n,1)=\M'(a-1)\C(n,b),\]
as required.
\end{proof}
%-----------------------------------
\begin{corollary}\label{I_m(n) by D(n,a) and M'}
Inside the $m\times\infty$ table, we have
\[\M'(a-1)\I_m(n)=2\parn{\frac{\M'(k-1)}{\M_m(0)}+\M'(k-2)+\cdots+\M'(0)}\C(n,a)\]
for all $a=1,\ldots,k$ and $n\geq1$, where $k=\ceil{m/2}$. In particular, for $a=1$, we have
\[\I_m(n)=2\parn{\frac{\M'(k-1)}{\M_m(0)}+\M'(k-2)+\cdots+\M'(0)}\C(n,1)\]
for all $n\geq1$.
\end{corollary}
\begin{proof}
We know that $\M'(a-1)\C(n,b)=\M'(b-1)\C(n,a)$ for all $b=1,\ldots,k$. If $m$ is odd, then
\begin{align*}
\M'(a-1)\I_m(n)&=\M'(a-1)(2\C(n,1)+\cdots+2\C(n,k-1)+\C(n,k))\\
&=(2\M'(0)+\cdots+2\M'(k-2)+\M'(k-1))\C(n,a)\\
&=2(\M'(0)+\cdots+\M'(k-2)+\M'(k-1)\M_m(0)^{-1})\C(n,a).
\end{align*}
Also, if $m$ is even, then
\begin{align*}
\M'(a)\I_m(n)&=\M'(a)(2\C(n,1)+\cdots+2\C(n,k-1)+2\C(n,k))\\
&=(2\M'(0)+\cdots+2\M'(k-2)+2\M'(k-1))\C(n,a)\\
&=2(\M'(0)+\cdots+\M'(k-2)+\M'(k-1)\M_m(0)^{-1})\C(n,a).
\end{align*}
The proof is complete.
\end{proof}

In order to find the recurrence relation of minimum degree among columns, we must to analyze the matrices formed by columns of the tables. The matrix of the $m\times\infty$ table is $\T_m$, which is the tridiagonal matrix with diagonal, superdiagonal, and subdiagonal entries are equal to $1$. Clearly, $\T_mC_m(n)=C_m(n+1)$ for all $n\geq1$, where $C_m(n)$ denotes the $n^{th}$ column. However, since the rows are symmetric, we can restrict ourself to the first $\ceil{m/2}$ rows. For this we define the following matrices according to the parity of $m$. Let
\[\O_1=[1],\quad\O_2=\begin{bmatrix}1&1\\2&1\end{bmatrix},\quad\O_k=\begin{bmatrix}1&1&0&\cdots&0&0&0\\1&1&1&\cdots&0&0&0\\&\vdots&&\ddots&&\vdots&\\0&0&0&\cdots&1&1&1\\0&0&0&\cdots&0&2&1\end{bmatrix}\quad(k\geq3),\]
and
\[\E_1=[2],\quad\E_2=\begin{bmatrix}1&1\\1&2\end{bmatrix},\quad\E_k=\begin{bmatrix}1&1&0&\cdots&0&0&0\\1&1&1&\cdots&0&0&0\\&\vdots&&\ddots&&\vdots&\\0&0&0&\cdots&1&1&1\\0&0&0&\cdots&0&1&2\end{bmatrix}\quad(k\geq3).\]
Also, let 
\[\T^*_m=\begin{cases}\O_{\ceil{\frac{m}{2}}},&m\text{ is odd},\\\E_{\ceil{\frac{m}{2}}},&m\text{ is even}\end{cases}\]
be the reduced matrix of the $m\times\infty$ table for all $m\geq1$.

Assume $C^*_m(n)$ is the reduced $n^{th}$ column in the $m\times\infty$ table including entries in the rows $1,\ldots,\ceil{m/2}$. From the definition, it follows that $\T^*_m C^*_m(n)=C^*_m(n+1)$ for all $n\geq1$.
%-----------------------------------
\begin{lemma}\label{det(T^*_m)}
For every $m\geq1$, we have
\[\det\T^*_m=\begin{cases}(-1)^{\floor{\frac{\ceil{\frac{m}{2}}+1}{3}}}2^{\chi_{3\Z}\parn{\ceil{\frac{m}{2}}}},&m\text{ is odd},\\(-1)^{\floor{\frac{\ceil{\frac{m}{2}}}{3}}}2^{\chi_{3\Z+1}\parn{\ceil{\frac{m}{2}}}},&m\text{ is even},\end{cases}\]
where $\chi$ denotes the characteristic function.
\end{lemma}
\begin{proof}
Expanding the determinants on the first row and then on the second row yields
\[\det(\O_k)=\det(\O_{k-1})-\det(\O_{k-2})\quad\text{and}\quad\det(\E_k)=\det(\E_{k-1})-\det(\E_{k-2})\]
for all $k\geq3$. A simple verification shows that the sequences $\{\det(\O_k)\}$ and $\{\det(\E_k)\}$ are periodic of length $6$ and begin with 
\[1,-1,-2,-1,1,2\quad\text{and}\quad2,1,-1,-2,-1,1,\]
respectively. Hence, the result follows.
\end{proof}
%-----------------------------------
\begin{corollary}
Let $k=\ceil{m/2}$. Then 
\[\det([C^*_m(n+1)\cdots C^*_m(n+k)])=\det(\T^*_m)^n\]
for all $n\geq0$.
\end{corollary}
\begin{proof}
It is not difficult to see, say by using induction on columns, that
\[[C^*_m(1)\cdots C^*_m(k)]=\begin{bmatrix}1&2&*&\cdots&*\\1&3&8&\cdots&*\\\vdots&\vdots&\vdots&\ddots&\vdots\\1&3&9&\cdots&3^k-1\\1&3&9&\cdots&3^k\end{bmatrix}\]
is a matrix whose $ij^{th}$ entry is $3^{j-1}$ if $i\geq j$, and it is $3^{j-1}-1$ if $j=i+1$. Subtracting $(r-1)^{th}$ row from $r^{th}$ for $r=k,k-1,\ldots,2$, respectively, we reach to an upper-triangular matrix with diagonal including of only $1$'s. This shows that $\det([C^*_m(1)\cdots C^*_m(k)])=1$. Therefore,
\[\det([C^*_m(n+1)\cdots C^*_m(n+k)])=\det({\T^*_m}^n[C^*_m(1)\cdots C^*_m(k)])=\det(\T^*_m)^n,\]
as required.
\end{proof}

Utilizing Lemma \ref{det(T^*_m)}, we obtain the following result for linear combinations of columns immediately. 
%-----------------------------------
\begin{theorem}\label{columns linear combination}
Let $k=\ceil{m/2}$ and $[\alpha_1\cdots\alpha_k]^T$ be the solution to the matrix equation 
\[[C^*_m(1)\cdots C^*_m(k)][\alpha_1\cdots\alpha_k]^T=[C^*_m(k+1)]\]
inside the $m\times\infty$ table. Then
\begin{equation}\label{columns linear combination 1}
\C(n+k,i)=\alpha_1\C(n,i)+\cdots+\alpha_k\C(n+k-1,i)
\end{equation}
for all $1\leq i\leq m$ and $n\geq1$. As a result,
\begin{equation}\label{columns linear combination 2}
\I_m(n+k)=\alpha_1\I_m(n)+\cdots+\alpha_k\I_m(n+k-1)
\end{equation}
for all $n\geq1$. Moreover, these recurrence relations are of the minimum degree $k$, that is, all other recurrence relations for columns entries and columns sums can be derived from these recurrence relations.
\end{theorem}
\begin{proof}
The equality of \eqref{columns linear combination 1} follows that of Corollary \ref{det(T^*_m)}, and the equality \eqref{columns linear combination 2} is a consequence of \eqref{columns linear combination 1}.

Now, we show that the recurrence relations \eqref{columns linear combination 1} and \eqref{columns linear combination 2} have minimum degree. The fact that the recurrence relation \eqref{columns linear combination 1} has minimum degree is obvious since the matrices $[C^*_m(n)\cdots C^*_m(n+k-1)]$ are invertible and their columns are linearly independent. We use this fact to show that the recurrence relation \eqref{columns linear combination 2} has minimum degree too. First observe that $3\I_m(n)=\I_m(n+1)+2\C(n,1)$ so that $(2-\Delta)\I_m(n)=2\C(n,1)$. If there exists a recurrence relation of degree $k'<k$ for $\I_m(n)$, then the same recurrence relation holds for $(2-\Delta)\I_m(n)$ and hence $\C(n,1)$ (and consequently $\C(n,i)$ for all $1\leq i\leq m$ by Lemma \ref{two rows equivalence by M'}) satisfies the same recurrence relation of degree $k'$, which is a contradiction. The proof is complete.
\end{proof}
%-----------------------------------
\begin{corollary}\label{equivalent minimal polynomials}
For every $m\geq1$, the following polynomials are equal:
\begin{itemize}
\item[(1)]$\det(xI-\T^*_m)$;
\item[(2)]$\M_m(k)(x-1)$;
\item[(3)]$x^k-[1\cdots x^{k-1}][C^*_m(1)\cdots C^*_m(k)]^{-1}[C^*_m(k+1)]$,
\end{itemize}
where $k=\ceil{m/2}$.
\end{corollary}
\begin{proof}
First we prove the equality of (1) and (2). Analogous to Lemma \ref{det(T^*_m)}, one can easily show that 
\[\det(xI-\O_k)=(x-1)\det(xI-\O_{k-1})-\det(xI-\O_{k-2})\]
and
\[\det(xI-\E_k)=(x-1)\det(xI-\E_{k-1})-\det(xI-\E_{k-2})\]
for all $k\geq3$. Since $\det(xI-\O_1)=x-1=\M_o(1)(x-1)$, $\det(xI-\O_2)=(x-1)^2-2=\M_o(2)(x-1)$, $\det(xI-\E_1)=x-2=\M_e(1)(x-1)$, and $\det(xI-\O_2)=(x-1)^2-(x-1)-1=\M_e(2)(x-1)$, it follows that $\det(xI-\O_k)=\M_o(k)(x-1)$ and $\det(xI-\E_k)=\M_e(k)(x-1)$ for all $k\geq1$. Thus,
\[\det(xI-\T^*_m)=\M_m(k)(x-1)\]
for all $m\geq1$.

Now, we show the equality of (1) and (3). Since $\T^*_mC^*_m(n)=C^*_m(n+1)$ and $\T^*_m$ satisfies its characteristic polynomial $\det(xI-\T^*_m)=0$, it follows that $\det((\Delta+1)I-\T^*_m)C^*_m(n)=0$, where $\Delta+1$ is the shift operator sending $C^*_m(n)$ to $C^*_m(n+1)$. Thus, $C^*_m(n)$ satisfies the recurrence relation $\det((\Delta+1)I-\T^*_m)C^*_m(n)=0$ of degree $k$. On the other hand, by Theorem \ref{columns linear combination}, $C^*_m(n)$ satisfies the recurrence relation 
\[((\Delta+1)^k-[1\cdots (\Delta+1)^{k-1}][C^*_m(1)\cdots C^*_m(k)]^{-1}[C^*_m(k+1)])C^*_m(n)=0\]
of minimum degree $k$ arising from the polynomial in (3). Hence, the two recurrence relations, having the same degrees and leading coefficients, must be identical, which implies that the polynomials (1) and (3) must be equal. The proof is complete.
\end{proof}
%-----------------------------------
\begin{corollary}\label{columns linear combination: short form}
The recurrence relations \eqref{columns linear combination 1} and \eqref{columns linear combination 2} are given by any of the formulas
\[\det((\Delta+1)I-\T_m^*)\C(n,i)=0\quad\text{or}\quad\M_m(k)\C(n,i)=0\]
and
\[\det((\Delta+1)I-\T_m^*)\I_m(n)=0\quad\text{or}\quad\M_m(k)\I_m(n)=0,\]
respectively, for all $i=1,\ldots,m$ and $n\geq1$, where $k=\ceil{m/2}$.
\end{corollary}
%==================================================
\section{The multipliers $\M_m$ and $\M'$}
This section is devoted to the study of the multipliers $\M_m$ and $\M'$. As we shall see in the sequel, the multiplies $\M_m$ and $\M'$ on the values $a+b$ and $ab$ ($a,b\geq1$) can be obtain by their values on $a$ and $b$, respectively. These results will be applied to show that these multipliers admit a factorization property like natural numbers.
%-----------------------------------
\begin{theorem}
For any $m,k\geq1$ and $1\leq a\leq k$, we have
\[\M'(a-1)\M_m(k-1)\equiv\M_m(k-a)\pmod{\M_m(k)}.\]
\end{theorem}
\begin{proof}
Let $m\geq1$. The result for $k=\ceil{m/2}$ follows from the facts that 
\[\M'(a-1)\C(n,1)=\C(n,a),\quad\M_m(k-a)\C(n,1)=\M(k-1)\C(n,a),\]
and $\M_m(k)$ is the polynomial of minimum degree satisfying $\M_m(k)\C(n,1)=0$ by Corollary \ref{columns linear combination: short form}. Since $\M_m$ depends only on the parity of $m$, and $k$ takes any number as long as $m$ takes all odd or even numbers, the result holds for any choice of $m$ and $k$.
\end{proof}
%-----------------------------------
\begin{theorem}\label{M_m(a+b) and M'(a+b)}
For any $m\geq1$ and $a,b\geq1$, we have
\[\M_m(a+b)=\M'(a)\M_m(b)-\M'(a-1)\M_m(b-1)\]
and
\[\M'(a+b)=\M'(a)\M'(b)-\M'(a-1)\M'(b-1).\]
\end{theorem}
\begin{proof}
The result follows by induction on $a$.
\end{proof}

Utilizing the same techniques used in the proof of Lemma \ref{two rows equivalence by M}, one can show the following general result.
%-----------------------------------
\begin{lemma}\label{general multiplier by M}
Let $m\geq1$, $k=\ceil{m/2}$, $0\leq a\leq k$, and $0\leq b\leq\min\{a,k-a\}$. Then 
\[\M_o(b)\C(n,a)=\C(n,a-b)+\C(n,a+b).\]
\end{lemma}
%-----------------------------------
\begin{theorem}\label{M_o(b)M_m(a) and M_o(b)M'(a)}
For any $m\geq1$ and $a\geq b\geq0$, we have
\[\M_o(b)\M_m(a)=\M_m(a+b)+\M_m(a-b)\]
and
\[\M_o(b)\M'(a)=\M'(a+b)+\M'(a-b).\]
\end{theorem}
\begin{proof}
We prove only the first equality for the second equality follows analogously. Since $\M_m$ depends only on the parity of $m$, one can choose $m$ such that $k:=\ceil{m/2}>a+b$. By Lemma \ref{two rows equivalence by M},
\[\M_m(a)\C(n,k)=\M(0)\C(n,k-a).\]
Now since $b\leq\min\{a,k-a\}$, Lemma \ref{general multiplier by M} yields
\begin{align*}
\M_o(b)\M_m(a)\C(n,k)&=\M(0)\M_o(b)\C(n,k-a)\\
&=\M(0)(\C(n,k-a-b)+\C(n,k-a+b)).
\end{align*}
On the other hand,
\[(\M_m(a+b)+\M_m(a-b))\C(n,k)=\M(0)(\C(n,k-a-b)+\C(n,k-a+b))\]
by Lemma \ref{two rows equivalence by M}. Thus, 
\[(\M_o(b)\M_m(a)-\M_m(a+b)-\M_m(a-b))\C(n,k)=0.\]
Since $\M_m(k)$ is the polynomial of minimum degree satisfying $\M_m(k)\C(n,k)=0$ and $\M_o(b)\M_m(a)-\M_m(a+b)-\M_m(a-b)$ is a polynomial of degree less than $k$, it follows that
\[\M_o(b)\M_m(a)-\M_m(a+b)-\M_m(a-b)=0,\]
as required. %The second equality can be proved analogously by invoking Lemma \ref{two rows equivalence by M'}.
\end{proof}

Extending $\M'$ on negative numbers yields $\M'(-1)=0$. Using this fact, we can prove the following result.
%-----------------------------------
\begin{theorem}\label{M_m(ab) and M'(ab)}
For any $a,b\geq1$ we have
\[\M_m(ab)=\M'(a-1)\parn{\M_o(b)}\M_m(b)-\M'(a-2)\parn{\M_o(b)}\M_m(0)\]
and
\[\M'(ab)=\M'(a-1)\parn{\M_o(b)}\M'(b)-\M'(a-2)\parn{\M_o(b)}\M'(0).\]
\end{theorem}
\begin{proof}
We prove only the first equality since the the proof of the second equality is similar. For $a=1$, we have 
\begin{align*}
\M_m(b)&=1\cdot\M_m(b)-0\cdot\M_m(0)\\
&=\M'(0)\parn{\M_o(b)}\M_m(b)-\M'(-1)\parn{\M_o(b)}\M_m(0).
\end{align*}
Also, for $a=2$, we obtain
\begin{align*}
\M_m(2b)+\M_m(0)&=\M_o(b)\M_m(b)
\end{align*}
by Theorem \ref{M_o(b)M_m(a) and M_o(b)M'(a)}, that is, 
\[\M_m(2b)=\M'(1)\parn{\M_o(b)}\M_o(b)-\M'(0)\parn{\M_o(b)}\M_m(0).\]
Now assume that $a\geq2$ and the result holds for $a$ and $a-1$. Applying the multiplier $\M_o(b)$ on both sides of 
\[\M_m(ab)=\M'(a-1)\parn{\M_o(b)}\M_m(b)-\M'(a-2)\parn{\M_o(b)}\M_m(0)\]
in conjunction with Theorem \ref{M_o(b)M_m(a) and M_o(b)M'(a)} yields
\begin{multline*}
\M_m((a+1)b)+\M_m((a-1)b)=\M_o(b)\M_m(ab)\\
=\M_o(b)\M'(a-1)\parn{\M_o(b)}\M_m(b)-\M_o(b)\M'(a-2)\parn{\M_o(b)}\M_m(0).
\end{multline*}
Since
\[\M_m((a-1)b)=\M'(a-2)\parn{\M_o(b)}\M_m(b)-\M'(a-3)\parn{\M_o(b)}\M_m(0),\]
it follows that
\begin{align*}
\M_m((a+1)b)=&(\Delta\M'(a-1)-\M'(a-2))\parn{\M_o(b)}\M_m(b)\\
&-(\Delta\M'(a-2)-\M'(a-3))\parn{\M_o(b)}\M_m(0)\\
=&\M'(a)\parn{\M_o(b)}\M_m(b)-\M'(a-1)\parn{\M_o(b)}\M_m(0),
\end{align*}
which is the result for $a+1$. The proof is complete.
\end{proof}

The above theorem results in a nice factorization formula for $\M_o$.
%-----------------------------------
\begin{theorem}{(Factorization theorem for $\M_o$)}\label{M_o factorization}
For all $a,b\geq1$, we have
\[\M_o(ab)=\M_o(a)\circ\M_o(b)=\M_o(b)\circ\M_o(a).\]
As a result, if $n=p_1^{a_1}\ldots p_k^{a_k}$ is the canonical factorization of $n$ into distinct primes $p_1,\ldots,p_k$, then 
\[\M_o(n)=\M_o(p_1)^{a_1}\ldots\M_o(p_k)^{a_k},\]
where all the products are the combination of functions.
\end{theorem}
\begin{proof}
Utilizing Theorems \ref{M_m(ab) and M'(ab)} and \ref{M_m(a+b) and M'(a+b)}, one observes that
\begin{align*}
\M_o(ab)&=\M'(a-1)\parn{\M_o(b)}\M_o(b)-\M'(a-2)\parn{\M_o(b)}\M_o(0)\\
&=\parn{\M'(a-1)\M_o(1)-\M'(a-2)\M_o(0)}(\M_o(b))\\
&=\M_o(a)(\M_o(b))=\M_o(a)\circ\M_o(b),
\end{align*}
from which the first equality follows. The second equality follows that of the first one by using induction on $n$.
\end{proof}

The multipliers $\M_e$ and $\M'$ are neither factorisable nor commuting as $\M_o$ was in the above theorem. However, a little modification of the argument gives uniform factorization formulas for all multipliers $\M_m$ and $\M'$ as follows: For any $m\geq1$ and prime $p$ we define the \textit{prime functions} $\F_{m,p}$ and $\F'_p$ on the set of multipliers $\{\M_m(n)\}$ and $\{\M'(n)\}$, respectively, as 
\[\F_{m,p}(\M_m(n))=\M'(p-1)(\M_e(n))\M_m(n)-\M'(p-2)(\M_e(n))\M_m(0)\]
and
\[\F'_p(\M'(n))=\M'(p-1)(\M_e(n))\M'(n)-\M'(p-2)(\M_e(n))\M'(0)\]
for all $n\geq1$. Using Theorem \ref{M_m(ab) and M'(ab)}, we can give a factorization of the multipliers $\M_m(n)$ and $\M'(n)$ into suitable prime functions.
%-----------------------------------
\begin{theorem}{(Uniform factorization theorem)}
Let $n=p_1^{a_1}\ldots p_k^{a_k}$ be the canonical factorization of $n$ into distinct primes $p_1,\ldots,p_k$. Then 
\[\M_m(n)=\F_{m,p_1}^{a_1}\ldots\F_{m,p_k}^{a_k}\M_m(1)\]
and
\[\M'(n)={\F'}_{p_1}^{a_1}\ldots{\F'}_{p_k}^{a_k}\M'(1),\]
where all the products are the combination of functions.
\end{theorem}
\begin{proof}
By Theorem \ref{M_m(ab) and M'(ab)}, we have
\begin{align*}
\M_m(pn)&=\M'(p-1)(\M_e(n))\M_m(n)-\M'(p-2)(\M_e(n))\M_m(0)\\
&=\F_{m,p}(\M_m(n))
\end{align*}
and
\begin{align*}
\M'(pn)&=\M'(p-1)(\M_e(n))\M'(n)-\M'(p-2)(\M_e(n))\M'(0)\\
&=\F'_p(\M'(n))
\end{align*}
for all $n\geq1$ and primes $p$. Hence, the result follows by induction on $n$.
\end{proof}

One notes that the multipliers $\M_m$ can be derived from $\M'$, which when combined with previous results, gives further properties of $\M'$.
%-----------------------------------
\begin{lemma}\label{M_e and M_o by M'}
For all $n\geq1$, we have
\[\M_e(n)=\M'(n)-\M'(n-1)\]
and
\[\M_o(n)=\M_e(n)+\M_e(n-1)=\M'(n)-\M'(n-2).\]
\end{lemma}

The above lemma can be used to simplify Corollary \ref{I_m(n) by D(n,a) and M}.
%-----------------------------------
\begin{corollary}\label{1+M_m(1)+...+M_m(n)}
For all $n\geq1$, we have
\[1+\M_m(1)+\cdots+\M_m(n)=\M'(n)+(\M_m(0)-1)\M'(n-1).\]
\end{corollary}
\begin{proof}
Let $\M_m^*(n):=1+\M_m(1)+\cdots+\M_m(n)$. Since $\M_e(n)=\M'(n)-\M'(n-1)$ by Lemma \ref{M_e and M_o by M'}, it follows immediately that $\M_e^*(n)=\M'(n)$. Using Lemma \ref{M_e and M_o by M'} once more, we get $\M_o(n)=\M_e(n)+\M_e(n-1)$ so that
\begin{align*}
\M_o^*(n)&=1+\M_o(1)+\cdots+\M_o(n)\\
&=(1+\M_e(1)+\cdots+\M_e(n))+(\M_e(0)+\M_e(1)+\cdots+\M_e(n-1))\\
&=\M_e^*(n)+\M_e^*(n-1)\\
&=\M'(n)+\M'(n-1),
\end{align*}
as required.
\end{proof}
%-----------------------------------
\begin{corollary}
Inside the $m\times\infty$ table, we have
\[\M_m(k-a)\I_m(n)=2(\M'(k-1)+(\M_m(0)-1)\M'(k-2))\C(n,a)\]
for all $a=1,\ldots,k$ and $n\geq1$, where $k=\ceil{m/2}$.
\end{corollary}

We conclude this section with providing the precise formulas for the multipliers $\M_m$ and $\M'$.
%-----------------------------------
\begin{theorem}
For all $n\geq1$, we have
\begin{align*}
\M_o(n)&=\sum_{i=0}^{\floor{\frac{n}{2}}}(-1)^i\bracket{\binom{n+1-i}{i}-\binom{n-1-i}{i-2}}\Delta^{n-2i},\\
\M_e(n)&=\sum_{i=0}^n(-1)^{\ceil{\frac{i}{2}}}\binom{n-\ceil{\frac{i}{2}}}{\floor{\frac{i}{2}}}\Delta^{n-i},\\
\M'(n)&=\sum_{i=0}^{\floor{\frac{n}{2}}}(-1)^i\binom{n-i}{i}\Delta^{n-2i}.
\end{align*}
\end{theorem}

It is worthwhile to mention that the polynomials $(1/2)\M_o(n)(2x)$ coincide with the Chebyshev polynomials of the first kind, and that the polynomials obtained from substituting coefficients of $\M_e(n)(x)$ and $\M'(n)(x)$ by their absolute values are indeed the Fibonacci polynomials $F_{n+1}(x)$ and Lucas polynomials $L_{n+1}(x)$, respectively.
%==================================================
\section{Singer cycles}
Cyclic subgroups of $GL(n,q)$ of order $q^n-1$ are known as \textit{Singer cycles}, where $GL(n,q)$ denotes the group of all invertible $n\times n$ matrices over the field $\F_q$ with $q$ elements. While it is easy to show the existence of Singer cycles, there is no direct formula to generate them. Our computations show that some of the invertible matrices $\E_n$ and $\O_n$, regarded as matrices with entries in $GF(q)$, have orders $q^n-1$ giving rise to Singer cycles in $GL(n,q)$. So the following question and conjectures arise naturally:
%--------------------------------------------------
\begin{question}
For which values of $n$ and $q$ the invertible matrix $\E_n\in GL_n(q)$ has order $q^n-1$?
\end{question}
%--------------------------------------------------
\begin{conjecture}
The matrix $\E_n\in GL_n(q)$ has order $q^n-1$ only if $q=2$ or $3$.
\end{conjecture}

The following table shows all $n\leq 500$ for which $\E_n$ has order $q^n-1$ when $q=2$ or $3$.

\begin{center}
Values of $n$ ($\leq500$) for which $\E_n$ has order $q^n-1$
\begin{tabular}{|c|c|}
\hline
$q$&$n$\\
\hline
$2$&$2$, $3$, $5$, $9$, $11$, $14$, $23$, $26$, $29$, $35$, $39$, $41$, $53$, $65$, $69$, $74$, $81$, $83$, $86$, $89$, $95$,\\
&$105$, $113$, $119$, $131$, $146$, $155$, $158$, $173$, $179$, $189$, $191$, $209$, $221$, $230$, $231$,\\
&$233$, $239$, $243$, $251$, $254$, $281$, $293$, $299$, $303$, $323$, $326$, $329$, $359$, $371$, $375$,\\
&$386$, $398$, $411$, $413$, $419$, $429$, $431$, $443$, $453$, $470$, $473$, $491$\\
\hline
$3$&$3$, $5$, $9$, $11$, $23$, $29$, $35$, $39$, $41$, $53$, $65$, $69$, $81$, $83$, $89$, $95$, $99$, $105$, $113$, $119$,\\
&$131$, $155$, $173$, $179$, $189$, $191$, $209$, $221$, $231$, $233$, $239$, $243$, $251$, $281$, $293$,\\
&$299$, $303$, $323$, $329$, $359$, $371$, $375$, $411$, $413$, $419$, $429$, $431$, $443$, $453$, $491$\\
\hline
\end{tabular}
\end{center}

For matrices $\O_n$ we pose the following conjecture determining all those matrices giving rise to Singer cycles.
%--------------------------------------------------
\begin{conjecture}
The matrix $\O_n\in GL_n(q)$ is invertible of order $q^n-1$ if and only if $q=3$ and $n$ is a power of two.
\end{conjecture}
%==================================================

\end{document}